\documentclass[11pt, a4paper]{article}
\usepackage{amsfonts}
\usepackage{mathrsfs}

\usepackage{latexsym}
\usepackage{graphicx}

\long\def\delete#1{}

\newtheorem{theorem}{Theorem}
\newtheorem{lemma}{Lemma}

\newtheorem{claim}{Claim}
\usepackage{amssymb}

\title{{\bf  Vertices in all minimum paired-dominating sets\\ of block graphs}
\thanks{Supported in part by National Natural
Science Foundation of China (Nos. 60673048 and 10471044 ) and
Shanghai Leading Academic Discipline Project (No. B407).}}

\author{{\bf  Lei Chen$^1$}
\  \  {\bf Changhong Lu$^{2}$}\footnote{Correspond author. E-mail:
chlu@math.ecnu.edu.cn}\ \ \ \ {\bf Zhenbing Zeng$^1$}\\
      \\ 
       $^1$ Shanghai Key Laboratory of Trustworthy Computing, \\
          East China Normal University, Shanghai, 200062, China \\ 
            \\
      $^2$ Department of Mathematics, \\
       East China Normal University,  Shanghai, 200241, China
       }
\date{}
\begin{document}
\openup 1.0\jot \maketitle

\begin{quote} {\bf Abstract}~Let $G=(V,E)$ be a simple graph without
isolated vertices. A set $S\subseteq V$ is a paired-dominating set
if every vertex in $V-S$ has at least one neighbor in $S$ and the
subgraph induced by $S$ contains a perfect matching. In this paper,
we present a linear-time algorithm to determine whether a given
vertex in a block graph is contained in all its  minimum
paired-dominating sets.

\end{quote}

\section{Introduction}
Let $G=(V,E)$ be a simple graph without isolated vertices. The
distance between $u$ and $v$ in $G$, denoted by $d_G(u,v)$, is the
minimum length of a path between $u$ and $v$ in $G$. For a vertex
$v\in V$, the {\sl neighborhood} of $v$ in $G$ is defined as
$N_G(v)=\{u\in V~|~uv\in E\}$ and the {\sl closed neighborhood} is
defined as $N_G[v]=N_G(v)\cup \{v\}$. The {\sl degree} of $v$,
denoted by $d_G(v)$, is defined as $|N_G(v)|$. We use $d(u,v)$ for
$d_G(u,v)$, $N(v)$ for $N_G(v)$, $N[v]$ for $N_G[v]$ and $d(v)$ for
$d_G(v)$ if there is no ambiguity. For a subset $S$ of $V$, the
subgraph of $G$ induced by the vertices in $S$ is denoted by $G[S]$
and $G-S$ denote the subgraph induced by $V-S$. A {\sl matching} in
a graph $G$ is a set of pairwise nonadjacent edges in $G$. A {\sl
perfect matching} $M$ in $G$ is a matching such that every vertex of
$G$ is incident to an edge of $M$. Some other notations and
terminology not introduced  in here can be found in \cite{West}.

Domination and its variations in graphs have been extensively
studied \cite{ht1,ht2}. A set $S\subseteq V$ is a {\sl
paired-dominating set} of $G$, denoted PDS, if every vertex in $V-S$
has at least one neighbor in $S$ and the induced subgraph $G[S]$ has
a perfect matching $M$. Two vertices joined by an edge of $M$ are
said to be paired in $S$. The {\sl paired-domination number},
denoted by $\gamma_{pr}(G)$, is the minimum cardinality of a PDS. A
paired-dominating set of cardinality $\gamma_{pr}(G)$ is called a
$\gamma_{pr}(G)$-set. The paired-domination was introduced by Haynes
and Slater \cite{hs1,hs}. There are many results on this problem
\cite{he,kang2,kang3,pm,sy,chen}.

The study of characterizing vertices contained in all various kinds
of minimum dominating set, such as dominating set, total dominating
set and paired-dominating set, has received considerable attention
(see \cite{cm},\cite{em}, \cite{mm}). 
Those results are all restricted in trees. In this paper, we will
extend the result in \cite{mm}  to block graphs, which contain trees
as its subclass. In fact, we give  a linear-time algorithm to
determine whether a given vertex in a block graph is contained in
all its minimum paired-dominating sets. If changing the pruning
rules and judgement rules in our algorithm, our method is also
available to determine whether a given vertex is contained in all
minimum (total) dominating sets of a block graph .

\section{Pruning block graphs}
Let $G=(V,E)$ be a simple graph. A vertex $v$ is a {\sl cut-vertex}
if deleting $v$ and all edges incident to it increases the number of
connected components. A {\sl block} of $G$ is a maximal connected
subgraph of $G$ without cut-vertices.    A {\sl block graph} is a
connected graph whose blocks are complete graphs. If every block is
$K_2$, then it is a tree.

Let $G=(V,E)$ be a block graph. As we know, every block graph not
isomorphic to complete graph has at least two {\sl end blocks},
which are blocks with only one cut-vertex. A vertex in $G$ is a {\sl
leaf} if its degree is one. If a vertex is adjacent to a leaf, then
we call it a {\sl support vertex}.

\begin{lemma}\label{lem1}\cite{mm}
Let $T$ be a tree of order at least three. If   $u$ is a leaf in
$T$, then there exists a $\gamma_{pr}(T)$-set not containing $u$.
\end{lemma}

For block graphs, we have the following generalized result. The
proof is almost same as that of  Lemma \ref{lem1}, so it is omitted.

\begin{lemma}\label{lem2}
Let $G$ be a block graph of order at least three. If  $u$ is not a
cut-vertex of $G$, then there exists a $\gamma_{pr}(G)$-set not
containing $u$.
\end{lemma}

If $G$ is a block graph with order two, then  every vertex is
contained in the only minimum paired-dominating set.  If $G$ is a
complete graph with order at least three, no vertex of $G$ is
contained in all minimum paired-dominating sets.  Thus, in here, we
assume that the block graph $G$ with at least one cut-vertex. Let
$r$ be the given vertex in $G$ and we want to determine whether $r$
is contained in every $\gamma_{pr}(G)$-set. By Lemma \ref{lem2}, it
is enough to assume that $r$ is a cut-vertex of $G$.

Our idea is to prune the original graph $G$ into a small block graph
$\tilde{G}$ such that the given vertex $r$ is contained in all
minimum paired-dominating sets of $G$ if and only if it is contained
in all minimum paired-dominating sets of $\tilde{G}$. To do this, we
first need a vertex ordering and follow this ordering we can prune
the original graph.  For a vertex $v\in V(G)$ and a block $B$, {\sl
the distance of $v$ and $B$}, denoted by $d(v,B)$, is defined as the
maximum of $d(u,v)$ for $u\in V(B)$. We say a block $B$ is {\sl
farthest from $v$} if $d(v,B)$ is maximum over all blocks. Note that
$B$ is an end block if $B$ is farthest from $r$. To find the vertex
ordering, in here, we need to define a vertex ordering connected
operation. Let $S=x_1,x_2,\cdots, x_s$ be a vertex ordering and
$T=u_1,u_2,\cdots,u_t$ be another vertex ordering. We use $S+T$ to
denote a new vertex ordering
$x_1,x_2,\cdots,x_s,u_1,u_2,\cdots,u_t$. Beginning with a block
farthest from $r$ and working recursively inward, we can find a
vertex order $v_1,v_2,\cdots,v_n$ as follows.

\vspace{4mm}

\noindent {\bf Procedure VO}\\
$S=\emptyset$; ($S$ is a vertex ordering.)\\
Let $r$ be a cut-vertex of $G$;\\
While ($G\not=\emptyset$) do\\
\hspace*{4mm} If ($G$ is a complete graph) then\\
\hspace*{10mm}Let $V(G)=\{u_1,u_2,\cdots,u_{a}=r\}$.
 $S=S+u_1,u_2,\cdots,u_{a}$; \\
\hspace*{10mm}$G=G-\{u_1,u_2,\cdots,u_{a}\}$;\\
\hspace*{4mm} else\\
 \hspace*{10mm} Let $B$ be an end block farthest from $r$ with
 $V(B)=\{u_1,u_2,\cdots,u_b,x\}$, where $x$ \\
  \hspace*{10mm} is the cut-vertex in $B$. $S=S+u_1,u_2,\cdots,u_b$;\\
 \hspace*{10mm} $G=G-\{u_1,u_2,\cdots,u_b\}$;\\
\hspace*{4mm} endif\\
 enddo\\
 Output $S$.

\vspace{4mm}

Let $v_1,v_2,\cdots,v_n=r$ be the vertex ordering of a block graph
$G$ which is obtained by procedure VO. We define the following
notations:\\
$(a)$ $F_G(v_i)=v_j$, $j=\max\{k~|~v_iv_k\in E,~k>i\}$. $v_j$ is
called the  {\sl father}  of $v_i$ and $v_i$ is a {\sl child} of
$v_j$. Obviously, $v_j$ must be a cut-vertex in $G$. We use $F(v_i)$ for $F_G(v_i)$ if there is no ambiguity.\\
$(b)$ $C_G(v_i)=\{v_j~|~F_G(v_j)=v_i\}$.  \\
$(c)$ For a block graph $G$, we define a rooted tree $T(G)$, whose
vertex set is $V(G)$, and $uv$ is an edge of $T(G)$ if and only if
$F_G(u)=v$. The root of $T(G)$ is $r$. Moreover let $T_{v}$ be a
subtree of $T(G)$ rooted at $v$. Every vertex in $T_{v}$ except $v$
is a descendant of $v$. For a vertex $v\in V(G)$, $D_G(v)$ denotes
the vertex set consisting of the descendants of $v$ in $T(G)$ and
$D_G[v]=D_G(v)\cup\{v\}$. That is, $D_G[v]=V(T_{v})$.

Except the vertex ordering, we also need a labeling function
$l(v)~:~V\rightarrow~\{\emptyset,r_1,r_2\}$ of each vertex $v$ to
help us to determine which vertices can be pruned. At first,
$l(v)=\emptyset$ for every vertex $v\in V$.

The following procedure can prune a big block graph $G$ into a small
block graph $\tilde{G}$ such that $r$ is contained in all minimum
paired-dominating sets of $G$ if and only if $r$ is contained in all
minimum paired-dominating sets of $\tilde{G}$.

\vspace{4mm}

\noindent {\bf Procedure PRUNE.} Prune a given block graph into a
small block graph.\\
{\bf Input} A block graph with at least one cut-vertex and
a vertex ordering $v_1,v_2,\cdots,v_n$ obtained by procedure VO.
For every vertex $v$, $l(v)=\emptyset$.\\
{\bf Output} A smaller block graph.\\
{\bf Method}\\
\hspace*{4mm}  $S=\emptyset$;\\
\hspace*{4mm}  For $i=1$ to $n-1$ do\\
\hspace*{12mm} If ($v_i\not\in S$) then\\
\hspace*{20mm} If ($l(v_i)=\emptyset$ and there is no child $v$ such
that $l(v)=r_1$ or $l(v)=r_2$) then\\
\hspace*{28mm} $l(F(v_i))=r_1$;\\
\hspace*{20mm} else if ($v_i$ satisfies the conditions of Lemma
\ref{lem3} or Lemma \ref{lem6} or Lemma \ref{lem7}) then\\
\hspace*{28mm} $G=G-D_G[v_i]$;\\
\hspace*{28mm} If ($d(v_i)=2$ and $|V(B_1)|=|V(B_2)|=2$ and
$C_G(F(v_i))=\{v_i\}$) then\\
\hspace*{28mm} (Where $B_1$ and $B_2$ are same as those in
Lemma \ref{lem6})\\
\hspace*{36mm} $S=S\cup \{F(v_i)\}$;\\
\hspace*{28mm} endif\\
\hspace*{20mm} else if ($v_i$ satisfies the conditions of Lemma
\ref{lem8}) then\\
\hspace*{28mm} $G=G-(D_G(v_i)-V(B'))$, where $B'$ is same as $B'$ in
Lemma \ref{lem8}.\\
\hspace*{20mm} else if ($v_i$ satisfies the conditions of Lemma
\ref{lem11} or Lemma \ref{lem13}) then\\
\hspace*{28mm} $G=G-(D_G(v_i)-D_G[u])$, where $u$ is same as $u$ in
Lemma \ref{lem11} and Lemma \ref{lem13}.\\
\hspace*{28mm} $l(v_i)=l(u)=r_2$;\hspace*{10mm} (*)\\
\hspace*{28mm} If ($d(v_i,r)=2$ and $|V(B_1)|=|V(B_2)|=2$ and
$C_G(F(v_i))=\{v_i\}$) then\\
\hspace*{28mm} (Where $B_1$ and $B_2$ are same as those in
Lemma \ref{lem13})\\
\hspace*{36mm} $S=S\cup \{F(v_i)\}$;\\
\hspace*{28mm} endif\\
\hspace*{20mm} endif\\
\hspace*{12mm} endif\\
\hspace*{4mm}  endfor\\
\hspace*{4mm}  Output $G$.

\vspace{4mm}

Next, we will prove the correctness of procedure PRUNE. Let $G_i$ be
a subgraph of the original graph $G$ after $v_i$ is considered and
$G_0=G$. It is clear that $G_i$ is a block graph for every $1\le
i\le n-1$. We define that $C_i(v)=C_G(v)\cap V(G_i)$,
$D_i(v)=D_G(v)\cap V(G_i)$ and $D_i[v]=D_G[v]\cap V(G_i)$ for $0\le
i\le n-1$. Note that at the $i$-th loop, the pruning vertices, for
example say $D_G[v_i]$, are $D_{i-1}[v_i]$ as $G$ is updated at each
step, i.e., $G=G_{i-1}$ at this time. It is enough to prove that $r$
is contained in all $\gamma_{pr}(G_{i-1})$-set if and only if $r$ is
contained in all $\gamma_{pr}(G_{i})$-set for $1\le i\le n-1$. If
$G_i=G_{i-1}$ for some $i$, then it is obviously true. When $v_i$ is
considered, let $R_j=\{v~|~v\in V(G_{i-1})~and~l(v)=r_j\}$ for
$j=1,2$.

\begin{lemma}\label{lem3}
When $v_i$ is a considering vertex such that $d(r,v_i)\ge 3$. If
$l(v_i)=\emptyset$, $(R_1\cup R_2)\cap C_{i-1}(v_i)\not=\emptyset$
and $G_{i-1}[R_1\cap C_{i-1}(v_i)]$ has a perfect matching, then $r$
is contained in all $\gamma_{pr}(G_{i-1})$-set if and only if $r$ is
contained in all $\gamma_{pr}(G_{i})$-set, where
$G_{i}=G_{i-1}-D_{i-1}[v_i]$.
\end{lemma}
\begin{proof}
Let $D_1=R_1\cap C_{i-1}(v_i)$, $D_2=R_2\cap D_{i-1}(v_i)$ and
$D=D_1\cup D_2$. In details, $D_1=\{u_1,u_2,\cdots,u_a\}$ and
$D_2=\{x_1,y_1,\cdots, x_b,y_b\}$, where $x_jy_j\in E$ and
$F(y_j)=x_j$ for $1\le j\le b$ for $1\le j\le b$ (see  the line
indicated (*) in the procedure PRUNE.) Then we obtain the following
claim.

\begin{claim}\label{claim4}
$\gamma_{pr}(G_{i-1})=\gamma_{pr}(G_{i})+|D|$.
\end{claim}
\begin{proof}
Any $\gamma_{pr}(G_{i})$-set can be extended to a PDS of $G_{i-1}$
by adding $D$. Thus $\gamma_{pr}(G_{i-1})\le
\gamma_{pr}(G_{i})+|D|$. For converse, let $S$ be a
$\gamma_{pr}(G_{i-1})$-set. If $y_j\not\in S$, then
$|D_{i-1}(y_j)\cap S|\ge 2$ and $S-D_{i-1}(y_j)\cup \{y_j,z_j\}$,
where $z_j$ is a child of $y_j$, is also a
$\gamma_{pr}(G_{i-1})$-set. Thus we may assume $y_j\in S$ and $w_j$
be its paired vertex. If $x_j\not\in S$, then $S-\{w_j\}\cup
\{x_j\}$ is also a $\gamma_{pr}(G_{i-1})$-set. If $x_j\in S$ and
$w_j\not=x_j$, let $x_j'$ is the paired vertex of $x_j$. Then
$x_j'=v_i$, otherwise $S-\{w_j,x_j'\}$ is a smaller PDS of
$G_{i-1}$. It is a contradiction. If $N(v_i)\subseteq S$, then
$S-\{v_i,w_j\}$ is a smaller PDS of $G_{i-1}$. Thus there is a
neighbor $v_i'$ of $v_i$ such that $v_i'\not\in S$. In this case,
$S-\{w_j\}\cup \{v_i'\}$ is also a $\gamma_{pr}(G_{i-1})$-set.
Therefore, we may assume that $D_2\subseteq S$ and every vertex in
$D_2$ is paired with another vertex in $D_2$.

With the similar argument, we may assume that $D_1\subseteq S$. Let
$u_j'$ is the paired vertex of $u_j$ and $CC=\{u_j'~|~u_j'\not\in
D_1\}$. If $CC=\emptyset$, then we do nothing. If $CC\not=\emptyset$
and $v_i\not\in CC$, then $S-CC$ is a smaller PDS of $G_{i-1}$, a
contradiction. Thus we assume $v_i\in CC$ and it is paired with
$u_1$. Since $G_{i-1}[D_1]$ has a perfect matching, there must be a
vertex in $D_1$, say $u_2$, such that $u_2'\in CC$. If
$N(v_i)\subseteq S$, then $S-\{u_2',v_i\}$ is a smaller PDS of
$G_{i-1}$. Thus there exists a neighbor $v_i'$ of $v_i$ such that
$v_i'\not\in S$. In this case, $S-CC\cup \{v_i,v_i'\}$ is a
$\gamma_{pr}(G_{i-1})$-set. Up to now, we may assume that
$D\subseteq S$ and every vertex in $D$ is paired with another vertex
in $D$.

If $v_i\not\in S$, then $S-D$ is a PDS of $G_i$. Thus
$\gamma_{pr}(G_i)\le |S|-|D|=\gamma_{pr}(G_{i-1})-|D|$. Therefore
$\gamma_{pr}(G_{i-1})=\gamma_{pr}(G_i)+|D|$. If $v_i\in S$, let
$v_i'$ be its paired vertex. If $v_i'\in C_{i-1}(v_i)$, then there
exists a neighbor $v_i''$ of $v_i$ such that $v_i''\not\in S\cup
C_{i-1}(v_i)$. Otherwise $S-\{v_i,v_i'\}$ is a smaller PDS of
$G_{i-1}$. Thus $S-\{v_i'\}\cup \{v_i''\}$ is a
$\gamma_{pr}(G_{i-1})$-set. So we assume that $v_i'\not\in
C_{i-1}(v_i)$. If $N(v_i')\subseteq S$, then $S-\{v_i,v_i'\}$ is a
smaller PDS of $G_{i-1}$. Thus there is a neighbor $v_i''$ of $v_i'$
such that $v_i''\not\in S$, in this case, $S-\{v_i\}\cup \{v_i''\}$
is a $\gamma_{pr}(G_{i-1})$-set not containing $v_i$. We may assume
$S$ is such a $\gamma_{pr}(G_{i-1})$-set. Then $S-D$ is a PDS of
$G_i$. Thus $\gamma_{pr}(G_i)\le |S|-|D|=\gamma_{pr}(G_{i-1})-|D|$.
Therefore $\gamma_{pr}(G_{i-1})=\gamma_{pr}(G_i)+|D|$.~~~$\Box$
\end{proof}\\

If there is a $\gamma_{pr}(G_i)$-set $S'$ such that $r\not\in S'$,
then let $S=S'\cup D$. By Claim \ref{claim4}, $S$ is a
$\gamma_{pr}(G_{i-1})$-set and $r\not\in S$. Therefore, if $r$ is
contained in all $\gamma_{pr}(G_{i-1})$-set, then $r$ is contained
in all $\gamma_{pr}(G_{i})$-set.

For converse, let $S$ be an arbitrary $\gamma_{pr}(G_{i-1})$-set and
$PD=S\cap D_{i-1}[v_i]$.

\begin{claim}\label{claim5}
$|D|\le |PD|\le |D|+2$.
\end{claim}
\begin{proof}
It is obvious that $|PD|\ge |D|$. Next, we prove $|PD|\le |D|+2$.
Let $v_i'$ be the father of $v_i$, i.e., $F(v_i)=v_i'$, and $B$ is a
block of $G_{i-1}$ containing $v_i$ and $v_i'$. We discuss it
according to the order of $B$.

\vspace{2mm}

{\bf Case 1:} $V(B)=\{v_i,v_i'\}$\\
If $|PD|\ge |D|+4$ and $|PD|$ is even, then $v_i',v_i''\not\in S$,
where $v_i''$ is the father of $v_i'$. Otherwise, $S-PD\cup D$ is a
smaller PDS of $G_{i-1}$. However, $S-PD\cup D\cup \{v_i',v_i''\}$
is also a smaller PDS of $G_{i-1}$, a contradiction. If $|PD|\ge
|D|+3$ and $|PD|$ is odd, then $v_i$ and $v_i'$ are paired in $S$.
If $N(v_i')\subset S$, then $S-PD-\{v_i'\}\cup D$ is a smaller PDS
of $G_{i-1}$. Thus there is a neighbor $w$ of $v_i'$ such that
$w\not\in S$, then $S-PD\cup D\cup \{w\}$ is also a smaller PDS of
$G_{i-1}$. It is  a contradiction.

\vspace{2mm}

{\bf Case 2:} $V(B)\not=\{v_i,v_i'\}$\\
Let $w$ be another vertex in $V(B)$. If $|PD|\ge |D|+4$ and $|PD|$
is even, then $w,v_i'\not\in S$, then $S-PD\cup D\cup \{v_i',w\}$ is
a smaller PDS of $G_{i-1}$. If $|PD|\ge |D|+3$ and $|PD|$ is odd,
then $v_i\in S$. If $w$ is the paired vertex of $v_i$, then there
exists a neighbor $w'$ of $w$ such that $w'\not\in S$. However,
$S-PD\cup D\cup \{w'\}$ is a smaller PDS of $G_{i-1}$. It is a
contradiction. If $v_i'$ is the paired vertex of $v_i$, with the
same argument to Case 1, we can also get a contradiction.~~~$\Box$
\end{proof}\\

By Claim \ref{claim5}, we have $|D|\le |PD|\le |D|+2$. We discuss
the following cases according to the size of $PD$.

\vspace{2mm}

{\bf Case 1:} $|PD|=|D|+2$\\
In this case, $(N(v_i)\cap V(G_i))\cap S=\emptyset$. If $|N(v_i)\cap
V(G_i)|\ge 2$, then let $S'=S-PD\cup \{w',w''\}$, where $w',w''\in
N(v_i)\cap V(G_i)$. By claim \ref{claim4}, $S'$ is a
$\gamma_{pr}(G_i)$-set. Then $r\in S'$. Since $d(v_i,r)\ge 3$, then
$r\in S$. If $|N(v_i)\cap V(G_i)|=1$, then $F(v_i),F(F(v_i))\not\in
S$, then let $S'=S-PD\cup \{F(v_i),F(F(v_i))\}$. By Claim
\ref{claim4}, $S'$ is a $\gamma_{pr}(G_i)$-set. Then $r\in S'$.
Since $d(v_i,r)\ge 3$, then $r\in S$.

\vspace{2mm}

{\bf Case 2:} $|PD|=|D|+1$\\
In this case, $v_i\in S$, let $\tilde{v}$ be its paired vertex. If
$N(\tilde{v})\subseteq S$, then $S-PD-\{\tilde{v}\}\cup D$ is a
smaller PDS of $G_{i-1}$. Thus there is a neighbor $w$ of
$\tilde{v}$ such that $w\not\in S$. Let $S'=S-PD\cup \{w\}$. by
Claim \ref{claim4}, $S'$ is a $\gamma_{pr}(G_i)$-set. Then $r\in
S'$. Since $d(v_i,r)\ge 3$, then $r\in S$.

\vspace{2mm}

{\bf Case 3:} $|PD|=|D|$\\
In this case, let $S'=S-PD$. Then by Claim \ref{claim4}, $S'$ is a
$\gamma_{pr}(G_i)$-set. Then $r\in S'$. Thus $r\in S$.~~~$\Box$
\end{proof}

\begin{lemma}\label{lem6}
When $v_i$ is a considering vertex such that $d(r,v_i)=2$. Let $B_1$
be the block containing $v_i$ and $F(v_i)$, and let $B_2$ be the
block containing $F(v_i)$ and $r$. Suppose $l(v_i)=\emptyset$,
$(R_1\cup R_2)\cap C_{i-1}(v_i)\not=\emptyset$ and $G_{i-1}[R_1\cap
C_{i-1}(v_i)]$ has a perfect matching. If $G_{i-1}$ satisfies one of
the following
conditions:\\
$(1)$ $|V(B_1)|\ge 3$;\\
$(2)$ $|V(B_1)|=2$ and $C_{i-1}(F(v_i))\not=\{v_i\}$;\\
$(3)$ $|V(B_1)|=2$, $C_{i-1}(F(v_i))=\{v_i\}$ and $|V(B_2)|\ge 3$.\\
Then $r$ is contained in all $\gamma_{pr}(G_{i-1})$-set if and only
if $r$ is contained in all $\gamma_{pr}(G_{i})$-set, where
$G_{i}=G_{i-1}-D_{i-1}[v_i]$.
\end{lemma}
\begin{proof}
We still use the notations in Lemma \ref{lem3}. With the same
argument to Claim \ref{claim4},
$\gamma_{pr}(G_{i-1})=\gamma_{pr}(G_i)+|D|$.

If there is a $\gamma_{pr}(G_i)$-set $S'$ such that $r\not\in S'$,
then let $S=S'\cup D$. Thus $S$ is a $\gamma_{pr}(G_{i-1})$-set and
$r\not\in S$. Therefore, if $r$ is contained in all
$\gamma_{pr}(G_{i-1})$-set, then $r$ is contained in all
$\gamma_{pr}(G_{i})$-set.

For converse, let $S$ be an arbitrary $\gamma_{pr}(G_{i-1})$-set and
$PD=S\cap D_{i-1}[v_i]$. With the similar argument to Claim
\ref{claim5}, $|D|\le |PD|\le |D|+2$. We discuss the following case
according to the size of $PD$.

\vspace{2mm}

{\bf Case 1:} $|PD|=|D|+2$\\
If $|V(B_1)|\ge 3$, then let $w$ be a vertex in $V(B_1)$ other than
$v_i$ and $F(v_i)$. Then $w,F(v_i)\not\in S$. Let $S'=S-PD\cup \{w,
F(v_i)\}$. Then $S'$ is a $\gamma_{pr}(G_{i})$-set. Since any new
added vertex is not $r$, then $r\in S$. If $|V(B_1)|=2$ and
$C_{i-1}(F(v_i))\not=\{v_i\}$, let $w$ be a child of $F(v_i)$ other
than $v_i$. It is obvious that $r, F(v_i)\not\in S$. If $w\not\in
S$, then $S'=S-PD\cup \{F(v_i),w\}$ is a $\gamma_{pr}(G_{i})$-set.
If $w\in S$ and $w'$ is its paired vertex, then there is a neighbor
$w''$ of $w'$ such that $w''\not\in S$. Then $S'=S-PD\cup
\{F(v_i),w''\}$ is a $\gamma_{pr}(G_{i})$-set. Thus $r\not\in S'$.
It contradicts that  $r$ is contained in all $\gamma_{pr}(G_i)$-set.
If $|V(B_1)|=2$, $C_{i-1}(F(v_i))=\{v_i\}$ and $|V(B_2)|\ge 3$, let
$w$ be a vertex in $V(B_2)$ other than $F(v_i)$ and $r$. Then
$\{r,F(v_i),w\}\cap S=\emptyset$. Let $S'=S-PD\cup \{w,F(v_i)\}$.
Then $S'$ is a $\gamma_{pr}(G_{i})$-set. However, $r\not\in S'$. It
contradicts that $r$ is contained in all $\gamma_{pr}(G_i)$-set.

\vspace{2mm}

{\bf Case 2:} $|PD|=|D|+1$\\
In this case, $v_i\in S$. Let $v_i'$ be the paired vertex of $v_i$,
then $v_i'\in V(B_1)$. Suppose $|V(B_1)|\ge 3$. If $v_i'\not=F(v_i)$
and $F(v_i)\not\in S$, then $S'=S-PD\cup \{F(v_i)\}$ is a
$\gamma_{pr}(G_{i})$-set. If $v_i'\not=F(v_i)$ and $F(v_i)\in S$,
then $v_i'$ is a cut-vertex of $G_{i-1}$. Otherwise,
$S-PD-\{v_i'\}\cup D$ is a smaller PDS of $G_{i-1}$. It is
impossible that $C_{i-1}(v_i')\subseteq S$. Thus there is a child
$w$ of $v_i'$ such that $w\not\in S$. $S'=S-PD\cup \{w\}$ is a
$\gamma_{pr}(G_{i})$-set. If $v_i'=F(v_i)$, let $w$ be a vertex in
$V(B_1)$ other than $v_i$ and $F(v_i)$. If $w\not\in S$, then
$S'=S-PD\cup \{w\}$ is a $\gamma_{pr}(G_{i})$-set. If $w\in S$, then
$w$ is a cut-vertex. If its paired vertex $w'\in C_{i-1}(w)$, then
there is a neighbor $w''$ of $w'$ such that $w''\not\in S$.
$S'=S-PD\cup \{w''\}$ is a $\gamma_{pr}(G_{i})$-set. If $w\in S$ and
its paired vertex $w'\in V(B_1)$, then $w'$ is also a cut-vertex. It
is impossible that $C_{i-1}(w)\subseteq S$, i.e., there is a child
$w''$ of $w$ such that $w''\not\in S$. $S'=S-PD\cup \{w''\}$ is a
$\gamma_{pr}(G_{i})$-set. In any case, $r\in S'$. On the other hand,
any new added vertex is not $r$. So $r\in S$.

Suppose $|V(B_1)|=2$ and $C_{i-1}(F(v_i))\not=\{v_i\}$. In this
case, $v_i$ and $F(v_i)$ are paired in $S$. Let $w$ be a child of
$F(v_i)$ other than $v_i$. If $w\not\in S$, then $S'=S-PD\cup \{w\}$
is a $\gamma_{pr}(G_{i})$-set. If $w\in S$, let $w'$ be its paired
vertex. Then there is a neighbor $w''$ of $w'$ such that $w''\not\in
S$. $S'=S-PD\cup \{w''\}$ is a $\gamma_{pr}(G_{i})$-set. In any
case, $r\in S'$. Since any new added vertex is not $r$, thus $r\in
S$.

Suppose $|V(B_1)|=2$, $C_{i-1}(F(v_i))=\{v_i\}$ and $|V(B_2)|\ge 3$.
In this case, $v_i$ and $F(v_i)$ are paired in $S$. Moreover,
$r\not\in S$, otherwise $S-PD-\{F(v_i)\}\cup D$ is a smaller PDS of
$G_{i-1}$. Let $w$ be a vertex in $V(B_2)$ other than $F(v_i)$ and
$r$. It is obvious that $w\not\in S$, then $S'=S-PD\cup \{w\}$ is a
$\gamma_{pr}(G_{i})$-set. Thus $r\not\in S'$. It contradicts that
$r$ is contained in all $\gamma_{pr}(G_i)$-set.

\vspace{2mm}

{\bf Case 3:} $|PD|=|D|$\\
In this case, $S'=S-PD$ is a $\gamma_{pr}(G_{i})$-set. Then $r\in S$
due to $r\in S'$. $\Box$
\end{proof}\\

If $d(v_i,r)=2$, $|V(B_1)|=2$, $C_{i-1}(F(v_i))=\{v_i\}$,
$|V(B_2)|=2$ and $v_i$ satisfies other conditions in Lemma
\ref{lem6}, then we can not prune $G_{i-1}$. We call $B_2$ the first
kind of TYPE-1 block containing $r$.

\begin{lemma}\label{lem7}
When $v_i$ is a considering vertex such that $d(r,v_i)=1$. Let $B$
be the block containing $v_i$ and $r$. Suppose $l(v_i)=\emptyset$,
$(R_1\cup R_2)\cap C_{i-1}(v_i)\not=\emptyset$ and $G_{i-1}[R_1\cap
C_{i-1}(v_i)]$ has a perfect matching. If $|V(B)|\ge 4$ or
$|V(B)|=3$ and every vertex in $V(B)$ is cut-vertex, then $r$ is
contained in all $\gamma_{pr}(G_{i-1})$-set if and only if $r$ is
contained in all $\gamma_{pr}(G_{i})$-set, where
$G_{i}=G_{i-1}-D_{i-1}[v_i]$.
\end{lemma}
\begin{proof}
We still use the notations in Lemma \ref{lem3}.  With the same
argument to Claim \ref{claim4},
$\gamma_{pr}(G_{i-1})=\gamma_{pr}(G_i)+|D|$.

If there is a $\gamma_{pr}(G_i)$-set $S'$ such that $r\not\in S'$,
then let $S=S'\cup D$. Thus $S$ is a $\gamma_{pr}(G_{i-1})$-set and
$r\not\in S$. Therefore, if $r$ is contained in all
$\gamma_{pr}(G_{i-1})$-set, then $r$ is contained in all
$\gamma_{pr}(G_{i})$-set.

For converse, let $S$ be an arbitrary $\gamma_{pr}(G_{i-1})$-set and
$PD=S\cap D_{i-1}[v_i]$. With the similar argument to Claim
\ref{claim5}, $|D|\le |PD|\le |D|+2$.

Suppose $|PD|=|D|+2$, then $N(v_i)\cap V(B)\cap S=\emptyset$.
Otherwise, $S-PD\cup D$ is a smaller PDS of $G_{i-1}$. Thus
$r\not\in S$. If $|V(B)|\ge 4$, let $w_1$ and $w_2$ be two vertices
other than $v_i$ and $r$. In this case, $S'=S-PD\cup \{w_1,w_2\}$ is
a $\gamma_{pr}(G_i)$-set. However, $r\not\in S'$. It contradicts
that $r$ is contained in all $\gamma_{pr}(G_{i-1})$-set. If
$|V(B)|=3$ and every vertex in $V(B)$ is cut-vertex, let $w$ be
another vertex in $V(B)$ other than $v_i$ and $r$. If there is a
child $w_1$ of $w$ such that $w_1\not\in S$, then $S-PD\cup
\{w,w_1\}$ is a $\gamma_{pr}(G_i)$-set not containing $r$. It is
also a contradiction. Otherwise, take any child of $w$, say $w_1$.
Suppose $w_2$ is the paired vertex of $w_1$. If $N(w_2)\subset S$,
then $S-PD-\{w_2\}\cup D\cup \{w\}$ is a smaller PDS of $G_i$. Thus
there is a neighbor $w_3$ of $w_2$ such that $w_3\not\in S$. Then
$S'=S-PD\cup D\cup \{w,w_3\}$ is a $\gamma_{pr}(G_i)$-set not
containing $r$. It is still a contradiction.

Suppose $|PD|=|D|+1$, then $v_i\in S$. If $r$ is paired with $v_i$,
then we have done. If $|V(B)|\ge 4$, let $w_1$ and $w_2$ are two
vertices other than $v_i$ and $r$. We assume $w_1$ is the paired
vertex of $v_i$. If $w_2\not\in S$, then $S'=S-PD\cup \{w_2\}$ is a
$\gamma_{pr}(G_i)$-set. If $w_2\in S$, let $w_3$ be its paired
vertex. If $w_2$ is not a cut-vertex, then $S-PD-\{w_2\}\cup D$ is a
smaller PDS of $\gamma_{pr}(G_i)$-set. Thus $w_2$ is a cut-vertex.
If $w_3\in C_{i-1}(w_2)$, then there is a neighbor $w_4$ of $w_3$
such that $w_4\not\in S$. $S'=S-PD\cup \{w_4\}$ is a
$\gamma_{pr}(G_i)$-set. If $w_3\in V(B)$, then $w_3$ is also a
cut-vertex and there is a child $w_4$ of $w_3$ such that $w_4\not\in
S$. $S'=S-PD\cup \{w_4\}$ is a $\gamma_{pr}(G_i)$-set. If $|V(B)|=3$
and every vertex in $V(B)$ is cut-vertex, let $w$ be another vertex
in $V(B)$ other than $v_i$ and $r$. In this case, $w$ is the paired
vertex of $v_i$. If there is a child $w_1$ of $w$ such that
$w_1\not\in S$, then $S'=S-PD\cup \{w_1\}$ is a
$\gamma_{pr}(G_i)$-set. Otherwise, take any child of $w$, say $w_1$,
and $w_2$ is its paired vertex. If $N(w_2)\subseteq S$, then
$S-PD-\{w_2\}\cup D$ is a smaller PDS of $G_{i-1}$. Thus there is a
neighbor $w_3$ of $w_2$ such that $w_3\not\in S$. Then $S'=S-PD\cup
\{w_3\}$ is a $\gamma_{pr}(G_i)$-set. In any case, $r\in S'$.
However, any new added vertex is not $r$. Thus $r\in S$.

If $|PD|=|D|$, then $S'=S-PD$ is a $\gamma_{pr}(G_i)$-set. Thus
$r\in S$ due to $r\in S'$.~~~$\Box$
\end{proof}\\

If $d(v_i,r)=1$, $|V(B)|=3$, there is a vertex in $V(B)$ which is
not cut-vertex and $v_i$ satisfies other conditions in Lemma
\ref{lem7}, then we can not prune $G_{i-1}$. We call $B$ the second
kind of TYPE-1 block containing $r$. If $d(v_i,r)=1$, $|V(B)|=2$ and
$v_i$ satisfies other conditions in Lemma \ref{lem7}, we call $B$
the first kind of TYPE-2 block containing $r$.

\begin{lemma}\label{lem8}
When $v_i$ is a considering vertex such that $l(v_i)=r_1$. Let $B'$
is an end block containing $v_i$ in $G_{i-1}$. If $G_{i-1}[R_1\cap
C_{i-1}(v_i)]$ has a perfect matching, then $r$ is contained in all
$\gamma_{pr}(G_{i-1})$-set if and only if $r$ is contained in all
$\gamma_{pr}(G_{i})$-set, where $G_i=G_{i-1}-(D_{i-1}(v_i)-V(B'))$.
\end{lemma}
\begin{proof}
Let $D_1=R_1\cap C_{i-1}(v_i)$, $D_2=R_2\cap D_{i-1}(v_i)$ and
$D=D_1\cup D_2$. Similar to Claim \ref{claim4}, we obtain the
following claim.

\begin{claim}\label{claim9}
$\gamma_{pr}(G_{i-1})=\gamma_{pr}(G_i)+|D|$.
\end{claim}

If there is a $\gamma_{pr}(G_i)$-set $S'$ such that $r\not\in S'$,
then let $S=S'\cup D$. By Claim \ref{claim9}, $S$ is a
$\gamma_{pr}(G_{i-1})$-set and $r\not\in S$. Therefore, if $r$ is
contained in all $\gamma_{pr}(G_{i-1})$-set, then $r$ is contained
in all $\gamma_{pr}(G_{i})$-set.

For converse, let $S$ be an arbitrary $\gamma_{pr}(G_{i-1})$-set and
$PD=S\cap (D_{i-1}(v_i)-V(B'))$.

\begin{claim}\label{claim10}
$|D|\le |PD|\le |D|+1$.
\end{claim}
\begin{proof}
If $|PD|\ge |D|+2$ and $|PD|$ is even. Since $|V(B')\cap S|\ge 1$,
then either $v_i\in S$ or $y\in S$, where $y\in V(B')-\{v_i\}$.
$S-PD\cup D$ is a smaller PDS of $G_{i-1}$. It is a contradiction.

If $|PD|\ge |D|+3$ and $|PD|$ is odd. In this case, $v_i\in S$ and
its paired vertex $v\in C_{i-1}(v_i)-V(B')$. Let $x\in
V(B')-\{v_i\}$, then $x\not\in S$. $S-PD\cup D\cup \{x\}$ is a
smaller PDS of $G_{i-1}$. It is a contradiction.~~~$\Box$
\end{proof}\\

If $|PD|=|D|+1$, then $v_i\in S$ and its paired vertex $v\in
C_{i-1}(v_i)-V(B')$. Let $S'=S-PD\cup \{x\}$, where $x\in
V(B')-\{v_i\}$. By Claim \ref{claim9}, $S'$ is a
$\gamma_{pr}(G_{i})$-set. Then $r\in S'$. Since $x\not=r$, $r\in S$.

If $|PD|=|D|$. Since $V(B')\cap S\not=\emptyset$, Thus $S'=S-PD$ is
a PDS of $G_i$. By Claim \ref{claim9}, $S'$ is also a
$\gamma_{pr}(G_{i})$-set. Thus $r\in S$ due to $r\in S'$.~~~$\Box$
\end{proof}

\begin{lemma}\label{lem11}
When $v_i$ is a considering vertex such that $d(r,v_i)\ge 3$ and
$G_{i-1}[R_1\cap C_{i-1}(v_i)]$ has not a perfect matching, let $M$
be the maximum matching in $G_{i-1}[R_1\cap C_{i-1}(v_i)]$ and $u\in
(R_1\cap C_{i-1}(v_i))-V(M)$. Then $r$ is contained in all
$\gamma_{pr}(G_{i-1})$-set if and only if $r$ is contained in all
$\gamma_{pr}(G_{i})$-set, where
$G_i=G_{i-1}-(D_{i-1}(v_i)-D_{i-1}[u])$.
\end{lemma}
\begin{proof}
Let $D_1=R_1\cap C_{i-1}(v_i)$ and $D_2=R_2\cap D_{i-1}(v_i)$. Take
one child of each vertex in $D_1-V(M)-\{u\}$ to construct vertex set
$D_1'$. $D=D_1\cup D_2\cup D_1'-\{u\}$. Then we obtain the following
claim.

\begin{claim}\label{claim12}
$\gamma_{pr}(G_{i-1})=\gamma_{pr}(G_i)+|D|$.
\end{claim}
\begin{proof}
Any $\gamma_{pr}(G_i)$-set can be extended to a PDS of $G_{i-1}$ by
adding $D$. Thus $\gamma_{pr}(G_{i-1})\le\gamma_{pr}(G_i)+|D|$.

For converse, let $S$ be a $\gamma_{pr}(G_{i-1})$-set. With the same
argument to Claim \ref{claim4}, $D_2\subset S$ and every vertex in
$D_2$ is paired with another vertex in $D_2$. Moreover, we may
assume $D_1\subset S$. Let $CC=\{x~|~x\not\in
D_1,~x~is~paired~with~one~vertex~in~D_1\}$. Since $M$ is a maximum
matching of $G_{i-1}[D_1]$. Thus $|CC|\ge |D_1|-|V(M)|=|D_1'|+1$. If
$v_i\not\in S$, then $S-CC\cup D_1'\cup \{v_i\}$ is also a
$\gamma_{pr}(G_{i-1})$-set. If $v_i\in S$ and $v_i$ is paired with
one vertex in $D_1$, then $S-CC\cup D_1'\cup \{v_i\}$ is also a
$\gamma_{pr}(G_{i-1})$-set. If $v_i\in S$ and $v_i$ is not paired
with any vertex in $D_1$, let $v$ be its paired vertex. Then
$v\not\in C_{i-1}(v_i)$, otherwise, $S-CC-\{v\}\cup D_1'$ is a
smaller PDS of $G_{i-1}$. Thus $v\in V(B)$, where $B$ is a block
containing $v_i$ and $F(v_i)$. If $N(v)\subseteq S$, then
$S-CC-\{v\}\cup D_1'$ is a smaller PDS of $G_{i-1}$. Thus there is a
neighbor $v'$ of $v$ such that $v'\not\in S$. Then $S-CC\cup
D_1'\cup \{v'\}$ is also a $\gamma_{pr}(G_{i-1})$-set. Therefore, we
may assume $D_1\cup D_1'\cup \{v_i\}\subseteq S$ and they are paired
each other. Since $u$ is the paired vertex of $v_i$, $S-D$ is a PDS
of $G_i$. Therefore, $\gamma_{pr}(G_i)\le
|S-D|=|S|-|D|=\gamma_{pr}(G_{i-1})-|D|$. So
$\gamma_{pr}(G_{i-1})=\gamma_{pr}(G_i)+|D|$.~~~$\Box$
\end{proof}\\

If there is a $\gamma_{pr}(G_i)$-set $S'$ such that $r\not\in S'$,
then let $S=S'\cup D$ if $u\in S'$ or $v_i\in S'$ and otherwise, let
$S=S'-D_{i-1}[u]\cup \{u,v_i\}\cup D$. By claim \ref{claim12}, $S$
is a $\gamma_{pr}(G_{i-1})$-set and $r\not\in S$. Therefore, if $r$
is contained in all $\gamma_{pr}(G_{i-1})$-set, then $r$ is
contained in all $\gamma_{pr}(G_{i})$-set.

For converse, let $S$ be an arbitrary $\gamma_{pr}(G_{i-1})$-set and
$PD=(D_{i-1}(v_i)-D_{i-1}[u])\cap S$. We obtain the following claim.

\begin{claim}
$|D|\le |PD|\le |D|+1$
\end{claim}
\begin{proof}
It is obvious that $|PD|\ge |D|$. Suppose $|PD|\ge |D|+2$ and $|PD|$
is even. If $v_i\in S$, then $S-PD\cup D$ is a smaller PDS of
$G_{i-1}$. If $v_i\not\in S$, then $S-D_{i-1}[v_i]\cup D\cup
\{v_i,u\}$ is a smaller PDS of $G_{i-1}$. It is a contradiction.
Suppose $|PD|\ge |D|+3$ and $|PD|$ is odd. In this case, one of
vertices $v_i,u$ is in $S$ such that its paired vertex is in
$D_{i-1}(v_i)-D_{i-1}[u]$. If $v_i$ is such a vertex, then
$|D_{i-1}[v_i]\cap S|\ge |D|+4$. $S-D_{i-1}[v_i]\cup D\cup
\{u,v_i\}$ is a smaller PDS of $G_{i-1}$. It is a contradiction. If
$u$ is such a vertex and $v_i\not\in S$, then $S-PD\cup D\cup
\{v_i\}$ is a smaller PDS of $G_{i-1}$. It is also a contradiction.
If $u$ is such a vertex and $v_i\in S$, then the paired vertex of
$v_i$ is not a child of $v_i$. Let $v$ be its paired vertex. If
$N(v)\subset S$, then $S-PD-\{v\}\cup D$ is a smaller PDS of
$G_{i-1}$. Thus there is a neighbor $v'$ of $v$ such that $v'\not\in
S$. However, $S-PD\cup D\cup \{v'\}$ is also a smaller PDS of
$G_{i-1}$. It is also a contradiction.~~~$\Box$
\end{proof}\\

Suppose $|PD|=|D|+1$. If $v_i\in S$ and its paired vertex is in
$D_{i-1}(v_i)-D_{i-1}[u]$, then $|D_{i-1}[v_i]\cap S|\ge |D|+2$. Let
$S'=S-D_{i-1}[v_i]\cup \{u,v_i\}$. By Claim \ref{claim12}, $S'$ is a
$\gamma_{pr}(G_i)$-set. If $u\in S$ and its paired vertex is in
$D_{i-1}(v_i)-D_{i-1}[u]$. If $v_i\not\in S$, then $S'=S-PD\cup
\{v_i\}$ is a $\gamma_{pr}(G_i)$-set by Claim \ref{claim12}. If
$v_i\in S$, let $v$ be its paired vertex. Then $v\in V(G_i)$ and
there is a neighbor $v'$ of $v$ such that $v'\not\in S$.
$S'=S-PD\cup D\cup \{v'\}$ is a $\gamma_{pr}(G_i)$-set. In any case,
$r\in S'$. Since $d(r,v_i)\ge 3$, any new added vertex is not $r$.
thus $r\in S$.

Suppose $|PD|=|D|$. If $v_i\not\in S$, then $S'=S-D_{i-1}[v_i]\cup
\{v_i,u\}$ is a $\gamma_{pr}(G_i)$-set. If $v_i\in S$, then
$S'=S-PD$ is a $\gamma_{pr}(G_i)$-set. In any case, $r\in S'$. Since
$d(v_i,r)\ge 3$, then $r\in S$.~~~$\Box$
\end{proof}\\

Similar to Lemma \ref{lem6} and Lemma \ref{lem7}, we can obtain the
following lemma. The detail of the proof is omitted in here.

\begin{lemma}\label{lem13}
When $v_i$ is a considering vertex such that $d(r,v_i)\le 2$ and
$G_{i-1}[R_1\cap C_{i-1}(v_i)]$ has not a perfect matching, let $M$
be the maximum matching in $G_{i-1}[R_1\cap C_{i-1}(v_i)]$ and $u\in
R_1\cap C_{i-1}(v_i)-V(M)$. Let $B_1$ be a block containing $v_i$
and $F(v_i)$ and $B_2$ be a block containing $F(v_i)$ and
$F(F(v_i))$ if exists. If
$G_{i-1}$ satisfies one of the following conditions:\\
$(1)$ $d(v_i,r)=2$ and $|V(B_1)|\ge 3$;\\
$(2)$ $d(v_i,r)=2$, $|V(B_1)|=2$ and $C_{i-1}(F(v_i))\not=\{v_i\}$;\\
$(3)$ $d(v_i,r)=2$, $|V(B_1)|=2$, $C_{i-1}(F(v_i))=\{v_i\}$ and $|V(B_2)|\ge 3$;\\
$(4)$ $d(v_i,r)=1$ and $|V(B_2)|\ge 4$;\\
$(5)$ $d(v_i,r)=1$, $|V(B_2)|=3$ and every vertex in $V(B_2)$ is cut-vertex.\\
Then $r$ is contained in all $\gamma_{pr}(G_{i-1})$-set if and only
if $r$ is contained in all $\gamma_{pr}(G_{i})$-set, where
$G_i=G_{i-1}-(D_{i-1}(v_i)-D_{i-1}[u])$.
\end{lemma}

If $d(v_i,r)=2$, $|V(B_1)|=2$, $C_{i-1}(F(v_i))=\{v_i\}$,
$|V(B_2)|=2$ and $v_i$ satisfies other conditions in Lemma
\ref{lem13}, then we can not prune $G_{i-1}$. We call $B_2$ the
first kind of TYPE-3 block containing $r$. If $d(v_i,r)=1$,
$|V(B_2)|=3$ and there is a vertex in $V(B_2)$ which is not
cut-vertex and $v_i$ satisfies other conditions in Lemma
\ref{lem13}, then we still can not prune $G_{i-1}$. We call $B_2$
the second kind of TYPE-3 block containing $r$. If $d(v_i,r)=1$,
$|V(B_2)|=2$ and $v_i$ satisfies other conditions in Lemma
\ref{lem13}, we call $B_2$ the second kind of TYPE-2 block
containing $r$.

\vspace{4mm}

Summarizing the above lemmas, we have

\begin{theorem}\label{thm14}
Let $G$ be a block graph with at least one cut-vertex and let
$\tilde{G}$ be the output of procedure PRUNE. Then $r$ is contained
in all minimum paired-dominating sets of $G$ if and only if $r$ is
contained in all minimum paired-dominating sets of $\tilde{G}$.
\end{theorem}

\section{Algorithm}
In this section, we will give some judgement rules to determine
whether $r$ is contained in all minimum paired-dominating sets of
$\tilde{G}$, where $\tilde{G}$ is the output of procedure PRUNE. Let
$\tilde{R}_j=\{v~|~v\in V(\tilde{G})~and~l(v)=r_j\}$ for $j=1,2$.
For $v\in V(\tilde{G})$, define $C_{\tilde{G}}(v)=C_G(v)\cap
V(\tilde{G})$, $D_{\tilde{G}}(v)=D_G(v)\cap V(\tilde{G})$ and
$D_{\tilde{G}}[v]=D_G[v]\cap V(\tilde{G})$.

According to lemmas in section 2, we can divide blocks containing
$r$ in $\tilde{G}$ into the following categories (suppose $B$ is a
block containing $r$ in $\tilde{G}$. Some examples of each category are shown in Fig. 1.):\\
$L_1=\{B~|~B$ is an end block with $|V(B)|=2\}$; \hspace{5mm}
$L_2=\{B~|~B$ is an end block with $|V(B)|\ge 3\}$;\\
$L_3=\{B~|~B$ is a TYPE-1 block$\}$; \hspace{24mm}
$L_4=\{B~|~B$ is a TYPE-2 block$\}$;\\
$L_5=\{B~|~B$ is a TYPE-3 block$\}$;\\
$L_6=\{B~|~|\tilde{R}_1\cap (V(B)-\{r\})|$ is odd
and $\tilde{R}_2\cap V(B)=\emptyset\}$;\\
$L_7=\{B~|~|\tilde{R}_1\cap (V(B)-\{r\})|\not=0$ is
even and $\tilde{R}_2\cap V(B)=\emptyset\}$;\\
$L_8=\{B~|~|\tilde{R}_1\cap (V(B)-\{r\})|$ is
odd and $\tilde{R}_2\cap V(B)\not=\emptyset\}$;\\
$L_9=\{B~|~|\tilde{R}_1\cap (V(B)-\{r\})|$ is even and
$\tilde{R}_2\cap V(B)\not=\emptyset\}$.

\begin{picture}(100,320)(-30,-80)
\put(-10,200){\circle*{4}}\put(-10,170){\circle*{4}}\put(-10,170){\line(0,1){30}}\put(-10,205){$r$}\put(-10,120){$L_1$}

\put(60,200){\circle*{4}}\put(90,170){\circle*{4}}\put(30,170){\circle*{4}}\put(60,200){\line(1,-1){30}}
\put(60,200){\line(-1,-1){30}}\put(30,170){\line(1,0){60}}\put(60,205){$r$}\put(60,120){$L_2$}

\put(130,200){\circle*{4}}\put(130,190){\circle*{4}}\put(130,180){\circle*{4}}\put(120,170){\circle*{4}}
\put(140,170){\circle*{4}}\put(120,160){\circle*{4}}\put(140,160){\circle*{4}}\put(150,170){\circle*{4}}
\put(150,160){\circle*{4}}\put(150,150){\circle*{4}}
\put(130,200){\line(0,-1){20}}\put(130,180){\line(1,-1){10}}\put(130,180){\line(-1,-1){10}}\put(120,170){\line(1,0){20}}
\put(120,170){\line(0,-1){10}}\put(140,170){\line(0,-1){10}}\put(130,180){\line(2,-1){20}}\put(150,170){\line(0,-1){20}}
\put(130,120){$L_3$}\put(130,205){$r$}\put(110,110){{\tiny(the first
kind)}}

\put(200,200){\circle*{4}}\put(190,190){\circle*{4}}\put(210,190){\circle*{4}}\put(180,180){\circle*{4}}
\put(200,180){\circle*{4}}\put(180,170){\circle*{4}}\put(200,170){\circle*{4}}\put(210,180){\circle*{4}}
\put(210,170){\circle*{4}}\put(210,160){\circle*{4}}
\put(200,200){\line(-1,-1){10}}\put(200,200){\line(1,-1){10}}\put(190,190){\line(1,0){20}}\put(190,190){\line(-1,-1){10}}
\put(190,190){\line(1,-1){10}}\put(180,180){\line(1,0){20}}\put(180,180){\line(0,-1){10}}\put(200,180){\line(0,-1){10}}
\put(190,190){\line(2,-1){20}}\put(210,180){\line(0,-1){20}}\put(200,120){$L_3$}\put(200,205){$r$}\put(180,110){{\tiny(the
second kind)}}

\put(270,200){\circle*{4}}\put(270,190){\circle*{4}}\put(260,180){\circle*{4}}
\put(280,180){\circle*{4}}\put(260,170){\circle*{4}}\put(280,170){\circle*{4}}\put(290,180){\circle*{4}}
\put(290,170){\circle*{4}}\put(290,160){\circle*{4}}
\put(270,200){\line(0,-1){10}}\put(270,190){\line(-1,-1){10}}
\put(270,190){\line(1,-1){10}}\put(260,180){\line(1,0){20}}\put(260,180){\line(0,-1){10}}\put(280,180){\line(0,-1){10}}
\put(270,190){\line(2,-1){20}}\put(290,180){\line(0,-1){20}}\put(270,120){$L_4$}\put(270,205){$r$}\put(250,110){{\tiny(the
first kind)}}

\put(340,200){\circle*{4}}\put(340,190){\circle*{4}}\put(330,180){\circle*{4}}
\put(350,180){\circle*{4}}\put(330,170){\circle*{4}}\put(360,180){\circle*{4}}
\put(360,170){\circle*{4}}\put(360,160){\circle*{4}}
\put(340,200){\line(0,-1){10}}\put(340,190){\line(-1,-1){10}}
\put(340,190){\line(1,-1){10}}\put(330,180){\line(1,0){20}}\put(330,180){\line(0,-1){10}}
\put(340,190){\line(2,-1){20}}\put(360,180){\line(0,-1){20}}\put(340,120){$L_4$}\put(340,205){$r$}\put(320,110){{\tiny(the
second kind)}}

\put(0,70){\circle*{4}}\put(0,60){\circle*{4}}\put(0,50){\circle*{4}}\put(-10,40){\circle*{4}}
\put(10,40){\circle*{4}}\put(-10,30){\circle*{4}}\put(20,40){\circle*{4}}
\put(20,30){\circle*{4}}\put(20,20){\circle*{4}}
\put(0,70){\line(0,-1){20}}\put(0,50){\line(1,-1){10}}\put(0,50){\line(-1,-1){10}}\put(-10,40){\line(1,0){20}}
\put(-10,40){\line(0,-1){10}}\put(0,50){\line(2,-1){20}}\put(20,40){\line(0,-1){20}}
\put(0,-10){$L_5$}\put(0,75){$r$}\put(-20,-20){{\tiny(the first
kind)}}

\put(70,70){\circle*{4}}\put(60,60){\circle*{4}}\put(80,60){\circle*{4}}\put(50,50){\circle*{4}}
\put(70,50){\circle*{4}}\put(50,40){\circle*{4}}\put(80,50){\circle*{4}}
\put(80,40){\circle*{4}}\put(80,30){\circle*{4}}
\put(70,70){\line(-1,-1){10}}\put(70,70){\line(1,-1){10}}\put(60,60){\line(1,0){20}}\put(60,60){\line(-1,-1){10}}
\put(60,60){\line(1,-1){10}}
\put(50,50){\line(1,0){20}}\put(50,50){\line(0,-1){10}}
\put(60,60){\line(2,-1){20}}\put(80,50){\line(0,-1){20}}\put(70,-10){$L_5$}\put(70,75){$r$}\put(50,-20){{\tiny(the
second kind)}}

\put(140,70){\circle*{4}}\put(120,50){\circle*{4}}\put(160,50){\circle*{4}}\put(120,30){\circle*{4}}
\put(140,70){\line(1,-1){20}}\put(140,70){\line(-1,-1){20}}\put(120,50){\line(1,0){40}}\put(120,50){\line(0,-1){20}}
\put(140,-10){$L_6$}\put(140,75){$r$}

\put(210,70){\circle*{4}}\put(190,50){\circle*{4}}\put(230,50){\circle*{4}}\put(190,30){\circle*{4}}
\put(230,30){\circle*{4}}
\put(210,70){\line(1,-1){20}}\put(210,70){\line(-1,-1){20}}\put(190,50){\line(1,0){40}}\put(190,50){\line(0,-1){20}}
\put(230,50){\line(0,-1){20}} \put(210,-10){$L_7$}\put(210,75){$r$}

\put(280,70){\circle*{4}}\put(260,50){\circle*{4}}\put(300,50){\circle*{4}}\put(260,30){\circle*{4}}
\put(300,30){\circle*{4}}\put(300,10){\circle*{4}}
\put(280,70){\line(1,-1){20}}\put(280,70){\line(-1,-1){20}}\put(260,50){\line(1,0){40}}\put(260,50){\line(0,-1){20}}
\put(300,50){\line(0,-1){40}} \put(280,-10){$L_8$}\put(280,75){$r$}

\put(350,70){\circle*{4}}\put(330,50){\circle*{4}}\put(370,50){\circle*{4}}\put(350,30){\circle*{4}}\put(330,30){\circle*{4}}
\put(370,30){\circle*{4}}\put(370,10){\circle*{4}}\put(350,10){\circle*{4}}
\put(350,70){\line(1,-1){20}}\put(350,70){\line(-1,-1){20}}\put(350,70){\line(0,-1){40}}\put(350,30){\line(1,1){20}}
\put(350,30){\line(-1,1){20}}\put(330,50){\line(1,0){40}}\put(330,50){\line(0,-1){20}}\put(350,30){\line(0,-1){20}}
\put(370,50){\line(0,-1){40}}\put(350,-10){$L_9$}\put(350,75){$r$}

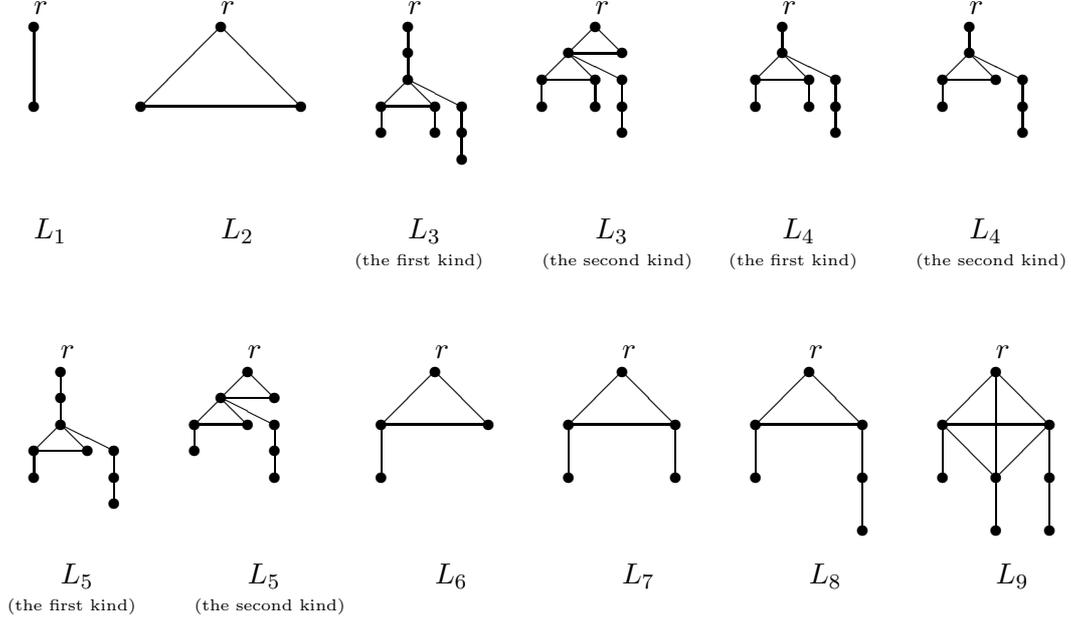
\put(40,-50){{\footnotesize Fig. 1. Some examples of nine categories
of blocks containing $r$ in $\tilde{G}$}}
\end{picture}

In order to simply the proof of judgement rules, we define $D(B)$
for any block $B\in \bigcup_{i=3}^9L_i$ as follows:\\
$(1)$: If $B\in L_3$, then $|V(B)|=2$ or $|V(B)|=3$ and there is a
vertex in $V(B)$ that is not cut-vertex. If $|V(B)|=2$, then $B$ is
the first kind. Let $u$ be the child of $r$ in $V(B)$ and $v$ be the
child of $u$. If $|V(B)|=3$, then $B$ is the second kind. Let $v$ be
the child of $r$ in $V(B)$ and $v$ is a cut-vertex. In any case,
$\tilde{G}[\tilde{R}_1\cap C_{\tilde{G}}(v)]$ has a perfect
matching. $D(B)=(\tilde{R}_1\cap
C_{\tilde{G}}(v))\cup (\tilde{R}_2\cap D_{\tilde{G}}(v))$. \\
$(2)$: If $B\in L_4$, then $|V(B)|=2$. Let $v$ be the child of $r$
in $V(B)$. If $B$ is the first kind, then $\tilde{G}[\tilde{R}_1\cap
C_{\tilde{G}}(v)]$ has a perfect matching. Let
$D(B)=(\tilde{R}_1\cap C_{\tilde{G}}(v))\cup (\tilde{R}_2\cap
D_{\tilde{G}}(v))$. Otherwise, let $M$ be the maximum matching in
$\tilde{G}[\tilde{R}_1\cap C_{\tilde{G}}(v)$. Take one child of each
vertex in $(\tilde{R}_1\cap C_{\tilde{G}}(v))-V(M)-\{w\}$ to
construct $D'$, where $w\in \tilde{R}_1\cap C_{\tilde{G}}(v)-V(M)$.
$D(B)=(\tilde{R}_1\cap C_{\tilde{G}}(v))\cup D'\cup (\tilde{R}_2\cap
D_{\tilde{G}}(v))\cup \{v,w\}$.\\
$(3)$: If $B\in L_5$, then $|V(B)|=2$ or $|V(B)|=3$ and there is a
vertex in $V(B)$ that is not cut-vertex. If $B$ is the first kind,
let $u$ be the child of $r$ in $V(B)$ and $v$ be the child of $u$.
If $B$ is the second kind, let $v$ be the child of $r$ in $V(B)$ and
$v$ is a cut-vertex. In any case, $\tilde{G}[\tilde{R}_1\cap
C_{\tilde{G}}(v)]$ has not a perfect matching. $D(B)$ is defined
same as the second kind of $(2)$.\\
$(4)$: If $B\in L_6\cup L_8$, let $CC=\bigcup_{v\in
V(B)}D_{\tilde{G}}[v]$. $D(B)=((\tilde{R}_1\cup \tilde{R}_2)\cap
CC)\cup \{w\}$, where $w$ is a child of some vertex
in $\tilde{R}_1\cap CC$.\\
$(5)$: If $B\in L_7\cup L_9$, let $CC=\bigcup_{v\in
V(B)}D_{\tilde{G}}[v]$. $D(B)=(\tilde{R}_1\cup \tilde{R}_2)\cap CC$.

\begin{lemma}\label{lem15}
Let $\tilde{G}$ be a output of procedure PRUNE, then $r$ is
contained in all minimum paired-dominating sets of $\tilde{G}$ if
and only if $\tilde{G}$ satisfies one of the following conditions:\\
$(1)$ $|L_1|\ge 1$;\\
$(2)$ $|L_1|=0$ and $|L_2|\ge 2$;\\
$(3)$ $|L_1|=0$, $|L_2|=1$ and $|L_3\cup L_6\cup L_8|\ge 1$;\\
$(4)$ $|L_1|=0$, $|L_2|=0$ and $|L_3|\ge 2$;\\
$(5)$ $|L_1|=0$, $|L_2|=0$, $|L_3|=1$ and $|L_6\cup L_8|\ge 1$.
\end{lemma}
\begin{proof}
If $|L_1|\ge 1$, then $r$ is a support vertex in $\tilde{G}$, and
hence $r$ is contained in all minimum paired-dominating sets of
$\tilde{G}$. Thus in the following discussion, we assume $|L_1|=0$.

\vspace{2mm}

{\bf Case 1:} $|L_2|\ge 2$\\
In this case, $r$ is contained in at least two end block with order
at least three, say $B_1$ and $B_2$ are two such blocks. Let $S$ be
an arbitrary $\gamma_{pr}(\tilde{G})$-set. If $r\not\in S$, then
$|V(B_i)\cap S|\ge 2$ for $i=1,2$. Then $S-V(B_1)-V(B_2)\cup
\{r,x\}$, where $x$ is a vertex in $V(B_1)-\{r\}$, is a smaller PDS
of $\tilde{G}$, a contradiction. Thus $r\in S$.

\vspace{2mm}

{\bf Case 2:} $|L_2|=1$ and $|L_3\cup L_6\cup L_8|\ge 1$\\
Let $B'\in L_2$ and $S$ be an arbitrary $\gamma_{pr}(\tilde{G})$-set
not containing $r$. It is obvious $|V(B')\cap S|\ge 2$. If $|L_3|\ge
1$, let $B\in L_3$. If $B$ is the first kind, let $u$ be a child of
$r$ in $V(B)$. Since $r\not\in S$, $|D_{\tilde{G}}[u]\cap S|\ge
2+|D(B)|$. However, $S-D_{\tilde{G}}[u]-V(B')\cup D(B)\cup \{r,u\}$
is a smaller PDS of $\tilde{G}$. If $B$ is the second kind, let $w$
be a vertex in $V(B)$ which is not cut-vertex and $u$ be another
vertex. Since $r\not\in S$, $|(D_{\tilde{G}}[u]\cup \{w\})\cap S|\ge
|D(B)|+2$. Then $S-D_{\tilde{G}}[u]-V(B')-\{w\}\cup D(B)\cup
\{r,u\}$ is a smaller PDS of $\tilde{G}$, a contradiction. Thus
$r\in S$.

If $|L_6\cup L_8|\ge 1$, let $B\in L_6\cup L_8$. $CC=\bigcup_{v\in
V(B)}D_{\tilde{G}}[v]$. Since $r\not\in S$, $|CC\cap S|\ge |D(B)|$.
However, $S-CC-V(B')\cup D(B)\cup \{r\}-\{w\}$, where $w\in D(B)$
and $l(w)=\emptyset$, is a smaller PDS of $\tilde{G}$, a
contradiction. Thus $r\in S$.

\vspace{2mm}

{\bf Case 3:} $|L_2|=1$ and $|L_3\cup L_6\cup L_8|=0$\\
Let $B'\in L_2$ and $y,z\in V(B')-\{r\}$. Since $r$ is a cut-vertex,
So $L_4\cup L_5\cup L_7\cup L_9\not=\emptyset$. Let $S'$ be a vertex
set by collecting $D(B)$ for any $B\in L_4\cup L_5\cup L_7\cup L_9$.
It is obvious that $S'\cup \{y,z\}$ is a
$\gamma_{pr}(\tilde{G})$-set. However, $r\not\in S$.

\vspace{2mm}

{\bf Case 4:} $|L_2|=0$ and $|L_3|\ge 2$\\
Let $B_1,B_2\in L_3$ and $S$ be an arbitrary
$\gamma_{pr}(\tilde{G})$-set. Suppose $r\not\in S$. For $B_j$
$(j=1,2)$, let $CC_j=\bigcup_{v\in V(B_j)}D_{\tilde{G}}[v]$. Since
$r\not\in S$, $|CC_j\cap S|\ge |D(B_j)|+2$ for $j=1,2$. However,
$S-CC_1-CC_2\cup D(B_1)\cup D(B_2)\cup \{r,u\}$, where $u$ is a
child of $r$ in $V(B_1)$, is a smaller PDS of $\tilde{G}$, a
contradiction. Thus $r\in S$.

\vspace{2mm}

{\bf Case 5:} $|L_2|=0$, $|L_3|=1$ and $|L_6\cup L_8|\ge 1$\\
Let $B_1\in L_3$ and $B_2\in L_6\cup L_8$. Suppose $S$ be an
arbitrary $\gamma_{pr}(\tilde{G})$-set and $r\not\in S$. For $B_j$
$(j=1,2)$, let $CC_j=\bigcup_{v\in V(B_j)}D_{\tilde{G}}[v]$. Since
$r\not\in S$, $|CC_1\cap S|\ge |D(B_1)|+2$ and $|CC_2\cap S|\ge
|D(B_2)|$. However, $S-CC_1-CC_2\cup D(B_1)\cup D(B_2)\cup
\{r\}-\{w\}$, where $w\in D(B_2)$ and $l(w)=\emptyset$, is a smaller
PDS of $\tilde{G}$,  a contradiction. Thus $r\in S$.

\vspace{2mm}

{\bf Case 6:} $|L_2|=0$, $|L_3|=1$ and $|L_6\cup L_8|=0$\\
Let $B\in L_3$. If $B$ is the first kind, let $u$ be the child of
$r$ in $V(B)$ and $v$ be the child of $u$. If $B$ is the second
kind, let $\{u,v\}=V(B)-\{r\}$. Let $S'$ be a vertex set by
collecting $D(B^*)$ for any $B^*\in L_4\cup L_5\cup L_7\cup L_9$.
Let $S=S'\cup D(B)\cup \{u,v\}$. Then it is obvious $S$ is a
$\gamma_{pr}(\tilde{G})$-set. However, $r\not\in S$.

\vspace{2mm}

{\bf Case 7:} $|L_2|=|L_3|=0$\\
Let $B$ be any block containing $r$, then $B\in L_4\cup L_5\cup
L_6\cup L_7\cup L_8\cup L_9$. Let $S'$ be a vertex set by collecting
$D(B)$ for any $B\in L_4\cup L_5\cup L_6\cup L_7\cup L_8\cup L_9$.
If $L_6\cup L_7\cup L_8\cup L_9\not=\emptyset$, then $S'$ is a
$\gamma_{pr}$-set of $\tilde{G}$. However, $r\not\in S'$. Thus we
may assume $L_6\cup L_7\cup L_8\cup L_9=\emptyset$. Then $B\in
L_4\cup L_5$. If there is a block $B\in L_4\cup L_5$ which is the
second kind of TYPE-2 or TYPE-3 block, then $S'$ is still a
$\gamma_{pr}$-set of $\tilde{G}$ not containing $r$. Thus we may
assume that $B\in L_4\cup L_5$ and $B$ is the first kind of TYPE-2
or TYPE-3 block. If there is a block $B\in L_5$, let $u$ be the
child of $r$ in $V(B)$ and $v$ is the child of $u$. Let $w$ be the
paired vertex in $D(B)$ and $w'$ be the child of $w$. Then $S=S'\cup
\{u,w'\}$ is a $\gamma_{pr}(\tilde{G})$-set of $\tilde{G}$. However,
$r\not\in S$. Then $B\in L_4$ for any block $B$ and $B$ is the first
kind of TYPE-2 block. Let $v$ be the child of $r$ in $V(B)$. If
there is a child $w$ of $v$ such that $l(w)=r_1$. Let $w'$ be the
child of $w$. Then $S=S'\cup \{v_1,w'\}$ is a
$\gamma_{pr}(\tilde{G})$-set not containing $r$. Thus we may assume
every child $w$ of $v$ satisfies $l(w)=r_2$. Let $w'$ be the child
of $w$ such that $l(w')=r_2$ and let $w''$ be the child of $w'$.
Take $S=S'\cup \{v,w''\}$. It is obvious that $S$ is a
$\gamma_{pr}(\tilde{G})$-set not containing $r$. ~~~$\Box$
\end{proof}\\

Now we are ready to present the algorithm to determine whether $r$
is contained in all minimum paired-dominating sets of $G$.

\vspace{4mm}

\noindent{\bf Algorithm VIAMPDS.} Determine whether the cut-vertex
$r$ of a block graph $G$ is contained in all minimum
paired-dominating sets of $G$\\
{\bf Input.} A block graph $G$ with at least one cut-vertex and a
cut-vertex $r$. The vertex ordering obtained by procedure VO.\\
{\bf Output.} True or False\\
{\bf Method}\\
\hspace*{4mm} Let $\tilde{G}$ be the output of procedure PRUNE with
input $G$.\\
\hspace*{4mm} Let $L_1=\{B~|~$ $B$ is an end block with $|V(B)|=2\}$;\\
\hspace*{12mm}$L_2=\{B~|~$ $B$ is an end block with
$|V(B)|\ge 3\}$;\\
\hspace*{12mm}$L_3=\{B~|~$ $B$ is a TYPE-1 block$\}$;\\
\hspace*{12mm}$L_6=\{B~|~$ $B$ is a block such that
$|\tilde{R}_1\cap (V(B)-\{r\})|$ is odd and $\tilde{R}_2\cap V(B)=\emptyset\}$;\\
\hspace*{12mm}$L_8=\{B~|~$ $B$ is a block such that
$|\tilde{R}_1\cap (V(B)-\{r\})|$ is odd and $\tilde{R}_2\cap
V(B)\not=\emptyset\}$.\\
\hspace*{12mm} ($B$ is a block containing $r$)\\
\hspace*{4mm}  If ($|L_1|\ge 1$) then\\
\hspace*{12mm} Return Ture;\\
\hspace*{4mm}  else if ($|L_2|\ge 2$) then\\
\hspace*{12mm} Return Ture;\\
\hspace*{4mm}  else if ($|L_2|=1$ and $|L_3\cup L_6\cup L_8|\ge 1$) then\\
\hspace*{12mm} Return Ture;\\
\hspace*{4mm}  else if ($|L_2|=0$ and $|L_3|\ge 2$) then\\
\hspace*{12mm} Return Ture;\\
\hspace*{4mm}  else if ($|L_2|=0$ and $|L_3|=1$ and $|L_6\cup L_8|\ge 1$) then\\
\hspace*{12mm} Return Ture;\\
\hspace*{4mm}  else\\
\hspace*{12mm} Return False;\\
\hspace*{4mm}  endif\\
\hspace*{4mm}  end\\

\begin{theorem}
Algorithm VIAMPDS can determine whether the give cut-vertex of a
block graph $G$ with at least one cut-vertex is contained in all
minimum paired-dominating sets in linear-time $O(n+m)$, where
$n=|V(G)|$ and $m=|E(G)|$.
\end{theorem}
\begin{proof}
By Theorem \ref{thm14}, $r$ is contained in all minimum
paired-dominating sets of $G$ if and only if $r$ is contained in all
minimum paired-dominating sets of $\tilde{G}$, where $\tilde{G}$ is
the output of procedure PRUNE with input $G$. Moreover, by Lemma
\ref{lem15}, the judgement rules in algorithm VIAMPDS can determine
whether $r$ is contained in all minimum paired-dominating sets of
$\tilde{G}$. On the other hand, every vertex and edge is used in a
constant times in algorithm VIAMPDS. Thus the theorem
follows.~~~$\Box$
\end{proof}

\section{Conclusion}

In this paper, we give a linear-time algorithm VIAMPDS to determine
whether the given vertex is contained in all minimum
paired-dominating sets of a block graph. Furthermore, the algorithm
VIAMPDS can be used to determine the set of vertices contained in
all minimum paired-dominating sets of a blocks graph in polynomial
time. Finally, we would like to point out that if changing the
pruning rules and judgement rules, our method is also available to
determine whether a given vertex is contained in all minimum (total)
dominating sets of a block graph.

\small {

}

\end{document}